\documentclass[12pt]{amsart}

\usepackage{amsmath,amssymb,amsthm,enumitem}

\usepackage{multirow}
\usepackage{tikz} 
\usepackage{amsfonts}
\usepackage{xcolor}

\usepackage{bm} 

\usetikzlibrary{matrix,arrows,decorations.pathmorphing,decorations.pathreplacing}

\usepackage{enumitem}
\usepackage{soul}

\usepackage{hyperref}

\makeatletter
\@namedef{subjclassname@2010}{%
  \textup{2010} Mathematics Subject Classification}
\makeatother

\DeclareMathOperator{\Gal}{Gal}

\DeclareMathOperator{\Num}{Num}
\DeclareMathOperator{\Jac}{Jac}
\DeclareMathOperator{\Gram}{Gram}

\def \Q {\mathbb Q}

\def \F {\mathbb F}
\renewcommand{\phi}{\varphi}

\usepackage{todonotes}

\newtheorem{theorem}{Theorem}[section]
\newtheorem*{thm}{Theorem}
\newtheorem*{prop}{Proposition}

\newtheorem{proposition}[theorem]{Proposition}

\newtheorem{corollary}[theorem]{Corollary}

\theoremstyle{definition}
\newtheorem{remark}[theorem]{Remark}
\newtheorem{definition}[theorem]{Definition}

\def\cqfd{
{\hfill
\kern 6pt\penalty 500
\raise -1pt\hbox{\vrule\vbox to 5pt{\hrule width 4pt
\vfill\hrule}\vrule}}
\break}

\font\tengoth=eufm10
\font\sevengoth=eufm7
\font\fivegoth=eufm5
\newfam\gothfam
\textfont\gothfam=\tengoth\scriptfont\gothfam=\sevengoth\scriptscriptfont\gothfam=\fivegoth

\usepackage[style=numeric,firstinits=true,hyperref,useprefix,backend=bibtex]{biblatex}
\bibliography{references}
\usepackage{csquotes}
\usepackage[normalem]{ulem}

\title[Maximal curves
 with respect to
  quadratic extensions]{Maximal curves with respect to quadratic extensions over
  finite fields
 }

\author[Aubry]{Yves Aubry}
\address[Aubry]{Institut de Math\'ematiques de Toulon - IMATH, Universit\'e de Toulon, France}
\email{yves.aubry@univ-tln.fr}
\address[Aubry]{Institut de Math\'ematiques de Marseille - I2M, Aix Marseille Univ, UMR 7373 CNRS, France}
\email{yves.aubry@univ-amu.fr}

\author[Herbaut]{Fabien Herbaut}
\address[Herbaut]{INSPE Nice-Toulon, Universit\'e C\^ote d'Azur, France}
\address[Herbaut]{Institut de Math\'ematiques de Toulon - IMATH, Universit\'e de Toulon, France}

\email{fabien.herbaut@univ-cotedazur.fr}

\author[Monaldi]{Julien Monaldi}
\address[Monaldi]{Institut de Math\'ematiques de Toulon - IMATH, Universit\'e de Toulon, France}
\email{julien.monaldi@ac-nice.fr}

\begin{document}

\begin{abstract}
We propose a detailed study
of 
a canonical bound
which relates  the numbers of rational points of a curve over a finite field 
with that over its quadratic extension.
Alternative proofs which 
make a connection  
with the variance
enable to 
obtain optimal refinements.

We focus on the curves reaching the bound,
which we call Hallouin-Perret-maximal curves.
We  provide different characterizations 
and stress natural links with the curves which attain the Ihara bound.
As consequences, we establish 
the list of 
such curves with low 
genus
and we 
outline a maximality result which involves the Suzuki curves.

At last
we determine 
which polynomials correspond to the Jacobian of a Hallouin-Perret-maximal curve of genus 2.
\end{abstract}

\date{\today}

\subjclass[2010]{Primary 11G20, 11G25, 11G10; Secondary 14G15, 14H25}

\keywords{Algebraic curves, finite fields, maximal curves, Diophantine stability, Abelian varieties}

\thanks{This work is partially supported by the French Agence Nationale de la Recherche through the Barracuda project under Contract ANR-21-CE39-0009-BARRACUDA}

\maketitle

\section{Introduction}

Throughout the whole paper we consider
an absolutely irreducible smooth projective algebraic curve $X$
(just called curve from now on) 
of genus $g$ and
defined over the finite field ${\mathbb F}_q$.
In the context of estimating the number  $\sharp X({\mathbb F}_{q^n})$ of rational points of $X$ over
${\mathbb F}_{q^n}$, 
we propose a
detailed study of an inequality
implicit in the work of Ihara in \cite{Ihara} and
hightlighted by Hallouin and Perret in \cite{Hallouin-Perret}.
Indeed, this inequality canonically appears in \cite{Hallouin-Perret}
among a series 
of meaningful bounds
obtained as consequences of non-negativity of
a series of Gram determinants.

Let us sketch the method developped in \cite{Hallouin-Perret}.
The Neron-Severi group 
 of the surface $X \times X$
can be quotiented by numerical equivalence
and thus tensorised
to obtain a real vector space $\Num(X \times X)_{\mathbb{R}}$
equipped by the intersection pairing.
As a consequence of the Hodge-index theorem
the intersection pairing is negative definite
 on the vector space orthogonal to the plane 
 generated by the classes
of the horizontal and vertical fibres.
We denote by $\wp$ the orthogonal projection onto this subspace.
For an integer $k$
we thus consider $\gamma_k$  the push-down by $\wp_*$
of the class
of the graph of the $k$-th iterated Frobenius that we normalize by $\sqrt{q}^k$.
The non-negativity of the determinant of the Gram matrix $\Gram(\gamma_0,\gamma_1)$
expresses exactly the Weil inequality, as explained in Subsection 2.2 in \cite{Hallouin-Perret}:
\begin{thm} (Weil, \cite{Weil})
Let $X$ be a curve defined over $\mathbb{F}_q$ of genus $g$. We have
\begin{equation}\label{Weilbound}
\mid \sharp X({\mathbb F}_q) - (q+1) \mid \ \leq 2g\sqrt{q}.
\end{equation}
\end{thm}
Next, the non-negativity of the Gram matrix $\Gram(\gamma_0,\gamma_1,\gamma_2)$
together with the arithmetic constraint 
$\sharp X({\mathbb F}_{q^2}) \geq \sharp X({\mathbb F}_q)$ yields to the Ihara bound,
as explained in Subsection 3.3 in \cite{Hallouin-Perret}:
\begin{thm} (Ihara, \cite{Ihara})
Let $X$ be a curve defined over $\mathbb{F}_q$ of genus $g \geq 1$. We have
\begin{equation}\label{Iharabound}
\sharp X({\mathbb F}_q) - (q+1) \ \leq \frac{ \sqrt{(8q+1)g^2+(4q^2-4q)g}-g}{2}.
\end{equation}
\end{thm}
Meanwhile, Hallouin and Perret
 have also noticed that the non-negativity of the  determinant of the Gram matrix 
$\Gram(\gamma_0,\gamma_1,\gamma_2)$ leads to the following inequality:
\begin{thm} (Hallouin and Perret, Proposition 12 in \cite{Hallouin-Perret})
Let $X$ be a curve defined over $\mathbb{F}_q$ of genus $g \geq 1$. We have
\begin{equation}\label{HP-bound}
\sharp X({\mathbb F}_{q^2}) - (q^2+1) \leq 2gq-\frac{1}{g}\bigl(\sharp X({\mathbb F}_{q}) - (q+1)\bigr)^2.
\end{equation}
\end{thm}

This is the inequality we aim to study in our work.
A first interpretation provided in \cite{Hallouin-Perret} is that the inequality (\ref{HP-bound})
improves the Weil bound for $\sharp X({\mathbb F}_{q^2})$
all the more as $\sharp X({\mathbb F}_q)$ is far from $\sharp \mathbb{P}^1({\mathbb F}_q)=q+1$.
Our first 
contribution in this paper
is to provide alternative and elementary proofs of the inequality (\ref{HP-bound}).
In particular we make a link with the statistical variance
of the real parts of the reciprocal roots of the $L$-polynomial of $X$
(that is of the numerator polynomial of the zeta function of $X$).
This way, the inequality (\ref{HP-bound}) appears as a consequence of the
 positivity
  of the variance.
As another consequence,
we can complete inequality (\ref{HP-bound})
with a 
lower bound.
\begin{thm}(Theorem \ref{delta}, Corollary \ref{intervalle} and Proposition \ref{improve_HP})
Let $X$ be a curve of genus $g\geq 2$ defined over ${\mathbb F}_q$.
We denote by $\alpha_1,\ldots, \alpha_g$  the real parts of its Frobenius eigenvalues,
that we consider  as a statistical sample
whose mean is given by
 $\displaystyle{{\overline{\alpha}}:=\frac{1}{g} \sum_{j=1}^g \alpha_j}$ and  whose variance equals
 $\displaystyle{V({\alpha}):=\frac{1}{g} \sum_{j=1}^g \left(\alpha_j-\overline{\alpha} \right)^2}$.
\begin{enumerate}[label=(\roman*)]
\item
The difference between the right hand side and the left hand side of the inequality (\ref{HP-bound})
 is a multiple 
of the variance of the $\alpha_j$'s:
\begin{equation*}
2gq-\frac{1}{g}\bigl(\sharp X({\mathbb F}_{q}) - (q+1)\bigr)^2 - \sharp X({\mathbb F}_{q^2}) + (q^2+1) = 4g V({\alpha}).
\end{equation*}
\item 
We deduce
\begin{equation*}
\left|  \# X(\mathbb{F}_{q^2}) - (q^2+1) + \frac{1}{g} \left( \# X(\mathbb{F}_q) -(q+1) \right)^2 \right|  \leq 2qg.
\end{equation*}
\item If $g$ is odd then we have
\begin{equation*} 
-2q \left(g-\frac{2}{g} \right) \leq  \# X(\mathbb{F}_{q^2}) - (q^2+1) + \frac{1}{g} \left( \# X(\mathbb{F}_q) -(q+1) \right)^2  .
\end{equation*} 
\end{enumerate}
\end{thm}
Let us now stress a link with a 
series of inequalities 
involving the  coefficients $a_1$ and $a_2$
 of a $q$-Weil polynomial
 $T^{2g}+a_1T^{2g-1}+a_2T^{2g-2}+\cdots+a_2 q^{g-2} T^{2} + a_1 q^{g-1} T+q^g $.
We recall that 
to be a $q$-Weil polynomial 
is a well-known necessary condition
to be the characteristic polynomial of an abelian variety defined over $\mathbb{F}_q$
as explained in Section \ref{subsection:qWeil}.

When  $g=2$
(respectively $g=3$, $g=4$, $g=5$)
R\"uck 
(respectively Haloui, Haloui and Singh, Sohn)
have stated an upper bound on the coefficient $a_2$
in \cite{Ruck} 
(respectively \cite{Haloui_JNT},
 \cite{Haloui-Singh},
 \cite{Sohn}).
One can also consult the work of Marseglia (see \cite{Marseglia})
for a survey and recent progresses on the subject. 
Actually, 
computations naturally related to the inequality (\ref{HP-bound})
lead to the following generalization of these four upper bounds for any value of $g \geq 2$
as well as to a lower bound on $a_2$.

\begin{thm}(Theorem \ref{th:bound_a_2} and Proposition \ref{somme_des_a_i})
We consider $g \geq 2$.
\begin{enumerate}[label=(\roman*)]
\item If $T^{2g}+a_1T^{2g-1}+a_2T^{2g-2}+\cdots+a_2 q^{g-2} T^{2} + a_1 q^{g-1} T+q^g $
is a $q$-Weil polynomial then
\begin{equation}\label{Borne_sur_a_2}
a_2 \leq \frac{a_1^2(g-1)}{2g} +gq \text{ and}
\end{equation}
\begin{equation}\label{equation:lower_bound_a2_intro}
a_2 \geq
\begin{cases}
\frac{g-1}{2g} a_1^2 -gq  &\  \text{when }  g \geq 2 \text{ is even}, \\
\frac{g-1}{2g} a_1^2 +\frac{(2-g^2)}{g}q  & \text{ when }  g  \geq 3 \text{ is odd}.
\end{cases}
\end{equation}
\item Let $X$ be a curve defined over ${\mathbb F}_q$ of genus $g\geq 2$ 
and
$L_X(T)=1+a_1T+a_2T^2+\cdots+a_2 q^{g-2} T^{2g-2} + a_1 q^{g-1} T^{2g-1}+q^g T^{2g}$
be its $L$-polynomial.
Then the inequalities (\ref{Borne_sur_a_2}) 
and
(\ref{equation:lower_bound_a2_intro}) obviously hold true. \\

\noindent Moreover the bound (\ref{Borne_sur_a_2}) is reached
 if and only (\ref{HP-bound}) is  an equality.
\end{enumerate}
\end{thm}

The case of equality in (\ref{HP-bound})
will thus receive special attention in our work.
Let us point out another meaningful motivation
to study this  case of equality.
One could relate with the 
words of Serre on page 96 in \cite{Serre}:
``it would be natural for curves to ask for many points, 
not only over ${\mathbb F}_q$,
 but also over several given extensions of ${\mathbb F}_q$''.
 Indeed, 
if we fix the values of $g$, $q$ and $N_1$,
 among the curves of genus $g>0$ with 
 $N_1$ points
 over $\mathbb{F}_q$,
 those which reach equality in (\ref{HP-bound})
have the largest number of points over ${\mathbb F}_{q^2}$.

\begin{definition}
A curve $X$ defined over ${\mathbb F}_q$ of genus $g>0$ is said to be a Hallouin-Perret-maximal curve (or HP-maximal curve for short) if
\begin{equation}\label{HP-equality}
\sharp X({\mathbb F}_{q^2}) - (q^2+1) = 2gq-\frac{1}{g}\bigl(\sharp X({\mathbb F}_{q}) - (q+1)\bigr)^2.
\end{equation}
\end{definition}

We will thus provide different characterizations 
of HP-maximal curves.
These characterizations will  prove useful
to identify families of HP-maximal curves
such as the Weil-maximal or Weil-minimal curves (see Example \ref{Weil-max-min}),
to establish the stability of 
the notion of HP-maximality
by coverings (see Proposition \ref{prop:coverings})
or to obtain an upper bound of the genus of a HP-maximal curve
depending only on $q$ (see Proposition \ref{prop:bound_g}).

\begin{prop}(Proposition \ref{Alpha} and Proposition \ref{prop:carac_geom})
Let $X$ be a curve of genus $g \geq 1$ defined over ${\mathbb F}_q$.
We denote by $\alpha_1,\ldots, \alpha_g$
the real parts of the reciprocal roots of the $L$-polynomial of $X$.
The following assertions are equivalent
\begin{enumerate}[label=(\roman*)]
\item  $X$  is a Hallouin-Perret-maximal curve, 
\item all the $\alpha_j$'s are equal, that is 
the zeta function of $X$ is of the form
$$Z_X(T)=\frac{(1-2\alpha T +qT^2)^g}{(1-T)(1-qT)}$$
(in this case, $2\alpha$ is an integer and we have $2\alpha=\frac{q+1-\sharp X({\mathbb F}_q)}{g}$),
\item the number of rational points of the Jacobian $\Jac(X)$
 of $X$ attains the upper bound (\ref{Bound_AHL}) given in \cite{AHL} 
$$\sharp \Jac(X)({\mathbb F}_q)= (q+1+\tau_{\Jac(X)}/g)^g$$
where $\tau_{\Jac(X)}$ stands for the opposite of the trace of the Jacobian of $X$, that is
 $
 \displaystyle{
 \tau_{\Jac(X)}:=-2\sum_{j=1}^g \alpha_j}$.
\end{enumerate}
\end{prop}

The inequality (\ref{HP-bound}) 
is strongly linked to the Ihara bound.
For instance, one can guess an elementary proof of (\ref{HP-bound})
in the original proof of Ihara of inequality (\ref{Iharabound})
based upon Cauchy-Schwarz inequality.
So it is quite natural that we can relate 
the curves reaching the two bounds.
To this end,
we recall that we have obviously 
$\sharp X({\mathbb F}_{q^2})\geq \sharp X({\mathbb F}_{q})$. 
We will follow \cite{M-R} 
and we will call
Diophantine-stable curve
(with respect to the extension 
${\mathbb F}_{q^2} / {\mathbb F}_q$)
a curve $X$ such that 
$\sharp X({\mathbb F}_{q^2})= \sharp X({\mathbb F}_{q})$. 

\begin{prop} (Proposition \ref{Ihara-max})
Let $X$ be a curve of genus $g \geq 1$. 
The curve $X$ is Ihara-maximal
(i.e. 
the Ihara inequality (\ref{Iharabound}) becomes an equality)
if and only if
$X$ is both a Hallouin-Perret-maximal curve and a Diophantine-stable curve with respect to the extension 
${\mathbb F}_{q^2} / {\mathbb F}_q$.
\end{prop}

We will obtain in Section \ref{section:Ihara}
the list of Ihara-maximal curves 
of low genera $g$ and also for low values of $q$.
Precisely, we give in Table \ref{table:ihara_g} the complete list up to isomorphism
of Ihara-maximal curves
when $g\leq 18$,
and in 
Remark \ref{Ihara_max_q_petit}
when $q\leq 13$,
except (in both cases) 
when $g=q=7$.
For these values, 
we know  
that there exists at least one Ihara-maximal curve,
but we do not know if there is unicity. \\

Whereas it is essentially 
a reformulation 
of a result of 
Fuhrmann and Torres (Theorem 2 in \cite{F-T})
we think it is worthwhile to 
highlight the
following theorem which provides
an analogue of a theorem of R\" uck and Stichtenoth 
where the Suzuki curves play the role of Hermitian curves
as explained in Subsection \ref{subsection:analog}.
\begin{thm} (Theorem \ref{theorem:suzuki_maximal}, reformulation of Theorem 2 in \cite{F-T})
We consider $t\geq1$ and $q=2^{2t+1}$. 
Let $X$ be a curve defined over ${\mathbb F}_q$.
Suppose that $X$ has genus $g=\frac{\sqrt{q}(q-1)}{\sqrt{2}}$.
Then
$X$ is Ihara-maximal 
if and only if $X$ is ${\mathbb F}_q$-isomorphic to the Suzuki curve $S$
which is the non-singular model of the curve of
 equation 
$y^q-y=x^{q_0}(x^q-x)$
where $q_0=2^t$.
\end{thm}

At last, it is noteworthy 
that the Jacobian of a HP-maximal curve is a power of a simple abelian variety
(see Proposition \ref{Power_Simple_Abelian_Variety}).
In the same direction,
we will determine 
among the polynomials $(T^2+aT+q)^2 \in \mathbb{Z}[T]$
which ones correspond to the Jacobian of a genus-$2$ HP-maximal curve $X$,
and we will characterize
when the Jacobian of $X$ is simple
or splits into the power of an ordinary or supersingular elliptic curve
(Theorem \ref{HP_curves_of_genus_2}).
This classification work will be the aim of
 Section \ref{section:structure}.
As a consequence we will deduce the following existence result.
\begin{prop}(Proposition \ref{prop:HPexistence})
Over any finite field   there exists a non-elliptic Hallouin-Perret-maximal curve.
\end{prop}
\section{Different characterizations of Hallouin-Perret-maximality}
Let $X$ be a curve defined over ${\mathbb F}_q$ of genus $g\geq 1$
and $N_n=\sharp X({\mathbb F}_{q^n})$ be its number of rational points over ${\mathbb F}_{q^n}$.
It is well-known that the zeta function of $X$
$$Z_{X,{\mathbb F}_q}(T)
:=
\exp\left(\sum_{n=1}^{\infty}\frac{N_nT^n}{n}\right)$$
is a rational function
$$Z_{X,{\mathbb F}_q}(T)=\frac{L_X(T)}{(1-T)(1-qT)}$$
where $L_X(T)$, called the $L$-polynomial of $X$, is a polynomial in ${\mathbb Z}[T]$ of degree $2g$.
It has the form
$$L_X(T)=\prod_{j=1}^g(1-\omega_j T)(1-\overline{\omega}_jT)$$
where the inverse roots $\omega_j$ of $L_X(T)$
are algebraic integers
such that $\vert \omega_j\vert = \sqrt{q}$ (the so-called Riemann Hypothesis for curves over finite fields, proved by Weil, see \cite{Weil}).
We also denote by $\alpha_j$ the real part of $\omega_j$, 
so that
 the $L$-polynomial of $X$ can be written 
$$L_X(T)=\prod_{j=1}^g(1-2\alpha_jT+qT^2)$$
where $\vert \alpha_j\vert \leq \sqrt{q}$.

\medskip

Let us recall now the link with the characteristic polynomial of the Jacobian of $X$.
Let $A$ be an abelian variety of dimension $g$ defined over ${\mathbb F}_q$.
The absolute Galois group 
$\Gal({\overline{\mathbb F}_q}/{\mathbb F}_q)$ 
is topologically generated by the automorphism $x\mapsto x^q$ 
and corresponds to
 the so-called
Frobenius endomorphism of $A$.
It induces an endomorphism on the Tate module $T_{\ell}(A)$ of $A$ (where $\ell$ is any prime number distinct from the characteristic of ${\mathbb F}_q$).
 The characteristic polynomial $f_A$ of $A$ 
 is defined as the characteristic polynomial of the Frobenius
  seen as an endomorphism on the ${\mathbb Q}_{\ell}$-vector space
   $T_{\ell}(A)\otimes_{{\mathbb Z}_{\ell}} {\mathbb Q}_{\ell}$.
It is a classic result that 
$f_A$ is a monic polynomial of degree $2g$ 
with integer coefficients whose roots 
have all modulus $\sqrt{q}$ (the Riemann Hypothesis for abelian varieties
 also proved by Weil, see \cite{Weil}).

It is also a standard result that 
the $L$-polynomial 
 of 
$X$ is 
the reciprocal polynomial of the characteristic polynomial 
of the Jacobian $\Jac(X)$ of $X$:

$$f_{\Jac(X)}(T)=T^{2g}L_X(1/T)=\prod_{j=1}^{2g}(T-\omega_j)(T-\overline{\omega}_j)=
\prod_{j=1}^g(T^2-2\alpha_j T+q).$$
 So we will sometimes refer to the $\omega_j$'s as  
 the Frobenius eigenvalues of  $X$.

\subsection{Characterization by the zeta function}
An important point in our work is the following proposition
which characterizes the HP-maximal curves
in terms of the real parts
of their Frobenius eigenvalues.
\begin{proposition}\label{Alpha}
Let $X$ be a curve of genus $g$ defined over ${\mathbb F}_q$.
We denote by $\alpha_1,\ldots, \alpha_g$
the real parts of
its  Frobenius eigenvalues.
The curve $X$  is a Hallouin-Perret-maximal curve 
if and only if  $\alpha_1=\cdots=\alpha_g$,
that is
if and only if
its zeta function is of the form
$$Z_X(T)=\frac{(1-2\alpha T +qT^2)^g}{(1-T)(1-qT)}$$
where $\alpha$ is the common value of all the $\alpha_j$.
In this case, $2\alpha$ is an integer and we have $2\alpha=\frac{q+1-\sharp X({\mathbb F}_q)}{g}$.   
\end{proposition}

\begin{proof}
It is well-known that we have:
$$\sharp X({\mathbb F}_q)=q+1-\sum_{j=1}^g (\omega_j + {\overline \omega_j})\ \ 
{\rm and}
\ \ \sharp X({\mathbb F}_{q^2})=q^2+1-\sum_{j=1}^g (\omega_j^2 + {\overline \omega_j}^2).$$
Saying that $X$ is a HP-maximal curve is thus equivalent to saying that 
$$-\sum_{j=1}^g (\omega_j^2 + {\overline \omega_j}^2)=2gq-\frac{1}{g}\bigl(-\sum_{j=1}^g (\omega_j + {\overline \omega_j})\bigr)^2.$$
But $\omega_j + {\overline \omega_j}=2\alpha_j$, so
$\omega_j^2 + {\overline \omega_j}^2=4\alpha_j^2-2q$ and thus the equality comes down to
$g\sum_{j=1}^g\alpha_j^2=\bigl(\sum_{j=1}^g \alpha_j\bigr)^2.$
But Cauchy-Schwarz equality holds if and only if
all the $\alpha_j$'s are equal,
that is if and only if $L_X$
has the claimed form.

 If $\alpha$ denotes the common value of the $\alpha_j$'s,
 by the relation above we get
$\alpha=\frac{q+1-\sharp X({\mathbb F}_q)}{2g}$.
Therefore
$2 \alpha$ is both a rational number and
an algebraic integer,
because it is the sum of the algebraic integers $\omega_1$ and $\overline{\omega}_1$.
We deduce that $2 \alpha$ is an integer.
\end{proof}

\begin{remark}\label{remark:elliptic}
It is straighforward that any elliptic curve $E$ over ${\mathbb F}_q$ is a HP-maximal curve since $N_1:= \sharp E({\mathbb F}_q)=q+1-(\omega+\overline{\omega})$ and
$N_2:=\sharp E({\mathbb F}_{q^2})=q^2+1-(\omega^2+\overline{\omega}^2)$ 
and thus
$(N_1-(q+1))^2=(\omega+\overline{\omega})^2
=q^2+1+2q-N_2$.
Meanwhile the characterization of Proposition \ref{Alpha} gives again immediately that any elliptic curve is a HP-maximal curve.
\end{remark}

\begin{remark}\label{Weil-max-min}
As a consequence of the previous proposition, 
a Weil-maximal curve
 i.e. a curve which reaches the  Weil upper bound
(respectively a Weil-minimal curve  i.e. a curve which reaches the  Weil lower bound)
 over ${\mathbb F}_q$ 
is 
a HP-maximal curve.
Indeed, if $X$ is a Weil-maximal curve over ${\mathbb F}_q$ of genus $g$ then 
$\sharp X({\mathbb F}_q)=q+1+2g\sqrt{q}$ and $q$ must be a square. 
If we still denote by $\omega_1, \overline{\omega}_1,\ldots,\omega_g,\overline{\omega}_g$ its Frobenius eigenvalues, then we can write $\sharp X({\mathbb F}_q)=q+1-\sum_{j=1}^g (\omega_j + {\overline \omega_j})$ and we thus obtain
that $\sum_{j=1}^g (\omega_j + {\overline \omega_j})=-2g\sqrt{q}$.
But the Riemann Hypothesis says that $\vert \omega_j\vert=\sqrt{q}$ for any $j=1,\ldots,g$ which implies that all the $\omega_j$ are equal to $-\sqrt{q}$
 and then $X$ is a HP-maximal curve.
In the same way, if $X$ is Weil-minimal then all the $\omega_j$ are equal to $\sqrt{q}$ and then $X$ is also a HP-maximal curve.
\end{remark}

\begin{remark}\label{2et20}
A HP-maximal curve of genus $g\geq 2$ 
is not necessarily Weil-maximal
nor Weil-minimal:
the curve $X$ of genus 2 defined over ${\mathbb F}_3$ by the equation $y^2=(-1-x-x^3)(1-x+x^3)$ 
verifies $N_1:=\sharp X({\mathbb F}_3)=2$
 and $N_2:=\sharp X({\mathbb F}_9)=20$
 so $X$ is a HP-maximal curve 
  but is neither
   Weil-maximal nor Weil-minimal.
\end{remark}

\subsection{Geometric condition}

\begin{proposition}\label{Power_Simple_Abelian_Variety}
Let $X$ be a curve of genus $g$ defined over ${\mathbb F}_q$.
If $X$ is a Hallouin-Perret-maximal curve then
the Jacobian 
of $X$
 is ${\mathbb F}_q$-isogenous to a power of a ${\mathbb F}_q$-simple abelian variety.
\end{proposition}

\begin{proof}
Let $f_{\Jac(X)}$ be the characteristic polynomial of the Jacobian of $X$
$$f_{\Jac(X)}=\prod_{j=1}^g(T^2-2\alpha_j T+q)$$
where the $\alpha_j$'s are the real parts of 
 its Frobenius eigenvalues
(and so $\vert \alpha_j\vert \leq \sqrt{q}$ by the Riemann Hypothesis).
By Proposition \ref{Alpha}
we know that if $X$ is a HP-maximal curve then
the $\alpha_j$'s are all equal.
If we denote by $\alpha$ this common value,  
it amounts to saying that $f_{\Jac(X)}$ has the form
$$f_{\Jac(X)}=(T^2-2\alpha T+q)^g.$$

But it is well-known that any abelian variety $A$ over ${\mathbb F}_q$ can factor uniquely, up to ${\mathbb F}_q$-isogeny, into a product of powers of non-${\mathbb F}_q$-isogenous ${\mathbb F}_q$-simple abelian varieties. 
And Tate stated in Theorem 2 (e) in
\cite{Tate} that $A$ is ${\mathbb F}_q$-isogenous to a power of a ${\mathbb F}_q$-simple abelian variety if and only if its characteristic polynomial is a power of a ${\mathbb Q}$-irreducible polynomial.

Since the  reduced discriminant of $T^2-2\alpha T+q$  is equal to $\alpha^2-q$, 
which is non-positive,
$f_{\Jac(X)}$ is  a power of a $\mathbb Q$-irreducible polynomial.
The result follows.
\end{proof}

\begin{remark}
The reciprocal
of Proposition \ref{Power_Simple_Abelian_Variety}
 is false
as we find many counterexamples in the database \cite{LMFDB}.
For instance, one can consider the curve $X$ of genus $2$
defined over $\mathbb{F}_{5}$ 
by the equation $y^2=x^5+3x$.
Its Jacobian is simple
as its characteristic polynomial $1+25x^4$ is irreducible in $\Q[x]$.
But it turns out that
we have $N_1=2$ and $N_2=26$ so $X$ is not a HP-maximal curve.
\end{remark}

\begin{remark}
The Jacobian of a Hallouin-Perret-maximal curve can be a simple abelian variety.
Indeed,
the LMFDB database \cite{LMFDB}
 provides\footnote{\url{https://www.lmfdb.org/Variety/Abelian/Fq/2/49/ao\_fr}}
the equations of many hyperelliptic curves of genus $2$
defined over ${\mathbb F}_{49}$
such that
$N_1=36$ and $N_2=2500$,
for example the curve of equation
$y^2
 =(3a+3)x^6 +(6a+1)x^5 +(2a+2)x^4 +(5a+6)x^3 +(5a+3)x^2 +3x+3a+4$
 where $a$ is a root of the Conway polynomial defining ${\mathbb F}_{49}$.
If we consider such a curve $X$,
one can easily check 
that the equality (\ref{HP-equality}) is verified, 
so $X$ is a HP-maximal curve. 
Moreover,
by Proposition \ref{Alpha} we recover that the $L$-polynomial of $X$ is
$L_X(T)=(1-7T+49T^2)^2.$
So,
if there existed
an elliptic curve $E$ defined over ${\mathbb F}_{49}$
with $L$-polynomial $L_E(T)=1-(\omega+\overline{\omega})T+qT^2=1-7T+49T^2$
we would have
$\sharp E({\mathbb F}_{49})=49+1-(\omega+\overline{\omega})
=43$.
But according 
to the LMFDB database 
once again,
such an elliptic curve does not exist.
So $X$ is an example of a HP-maximal curve whose Jacobian is a simple abelian surface.
\end{remark}

\begin{remark}
Let $X$ be a HP-maximal curve of genus $g$ defined over ${\mathbb F}_q$ with $q=p^a$ and let $f_{\Jac(X)}=(T^2-2\alpha T+q)^g$ be its characteristic polynomial 
where $\alpha=\frac{q+1-\sharp X({\mathbb F}_q)}{2g}$.
The Deuring-Waterhouse theorem (Theorem 4.1 in \cite{Waterhouse})
characterizes
the isogeny classes 
of elliptic curves 
over a finite field 
in terms of their characteristic polynomial.
It enables to
state that
 the Jacobian of $X$ is ${\mathbb F}_q$-isogenous to a power of an elliptic curve 
 if and only if one
of the following conditions is satisfied:
\begin{enumerate}[label=(\roman*)]
\item $(2\alpha,p)=1$ ;
\item $a$ is even and $2\alpha=\pm 2\sqrt{q}$ ;
\item $a$ is even and $p\not\equiv 1 \pmod 3$ and $2\alpha=\pm\sqrt{q}$ ;
\item $a$ is odd and $p=2$ or 3 and $2 \alpha=\pm p^{\frac{a+1}{2}}$ ;
\item either (v.i) $a$ is odd  or (v.ii) $a$ is even and $p\not\equiv 1\pmod 4$ and $\alpha=0$.
\end{enumerate}
\end{remark}

\subsection{Characterization by the number of points of the Jacobian}

Let $A$ be an abelian variety of dimension $g$ defined over ${\mathbb F}_q$
 with characteristic polynomial $f_A$ 
whose roots are denoted by
 $\omega_1, \overline{\omega}_1,\ldots,\omega_g,\overline{\omega}_g$.

If we set  $\tau_A:= -\sum_{j=1}^g(\omega_j+\overline{\omega}_j)=-2\sum_{j=1}^g \alpha_j$
 for the opposite of the trace of the Frobenius on $A$ (where the $\alpha_j$'s 
 are the real parts of the complex numbers $\omega_j$'s), then 
Aubry, Haloui and Lachaud proved in  Theorem 2.1 of \cite{AHL} 
the following bound on
  the number of rational points on  $A$: 
 \begin{equation}\label{Bound_AHL}
\sharp A({\mathbb F}_q)\leq (q+1+\tau_A/g)^g
\end{equation}
with equality if and only if 
the $\alpha_j$'s are equal.
 Theorem 2.1 of \cite{AHL} together with Proposition \ref{Alpha} lead  to the following statement.

\begin{proposition}\label{prop:carac_geom}
Let $X$ be a curve of genus $g$ defined over ${\mathbb F}_q$.
Then $X$ is a Hallouin-Perret-maximal curve 
if and only if
 the number of rational points of the Jacobian $\Jac(X)$
 of $X$ attains the upper bound (\ref{Bound_AHL}), namely 
\begin{equation}
\sharp \Jac(X)({\mathbb F}_q)= (q+1+\tau_{\Jac(X)}/g)^g.
\end{equation}
\end{proposition}

\begin{proof}
Let $X$ be  a curve of genus $g$ defined over ${\mathbb F}_q$ 
and let $\alpha_1,\ldots, \alpha_g$ be
the real parts of the inverse roots $\omega_j$ of its $L$-polynomial $L_X(T)$.
Now consider the Jacobian $\Jac(X)$ of $X$
 whose characteristic polynomial 
 $f_{\Jac(X)}$ 
 is the reciprocal polynomial of the $L$-polynomial $L_X(T)$
 of $X$.
By  Proposition \ref{Alpha} we have that $X$ is a HP-maximal curve 
if and only if 
the $\alpha_j$ are all equal.
Now Theorem 2.1 of \cite{AHL} ensures that it is equivalent to saying
that  $\Jac(X)$
attains the upper bound (\ref{Bound_AHL}).
\end{proof}

\subsection{Genus of a Hallouin-Perret-maximal curve}
As a consequence of Ihara's results in \cite{Ihara}
a curve defined over $\mathbb{F}_q$ cannot be Weil-maximal if its genus is large with respect to $q$
(see also Theorem 5.1.1 in \cite{Serre}).
Precisely, if $X$ is a Weil-maximal curve defined over ${\mathbb F}_q$ of genus $g$ then $g\leq \frac{q-\sqrt{q}}{2}$.
In  Proposition \ref{prop:bound_g} 
we will obtain that the genus of
a HP-maximal curve 
defined over ${\mathbb F}_q$
is also bounded by a function of $q$.

The Frobenius angles of a curve are defined to be
the arguments of the complex inverse roots $\omega_j$ of the $L$-polynomial
of the curve.
Elkies, Howe and Ritzenthaler 
have given in \cite{EHR} 
an explicit upper bound 
on the genus of
a curve
depending on the set of the Frobenius angles.
More precisely they have proved 
in Theorem 1.1 of \cite{EHR}
that if a non-empty set 
$S \subset [0,\pi]$ of cardinal $s$ 
contains all non-negative Frobenius angles of $X$,
then the genus $g$ of $X$
satisfies $g\leq 23 s^2 q^{2s}\log q$ and $g<(\sqrt{q}+1)^{2r}(\frac{1+q^{-r}}{2})$
where  $r=1/2$ if $S=\{0\}$ and 
$r=\sharp(S\cap\{\pi\})+2\sum_{\theta\in S\setminus\{0,\pi\}} \lceil\frac{\pi}{2\theta}\rceil$ otherwise.

As a consequence we obtain the following bounds.

\begin{proposition}{\label{prop:bound_g}}
Let $X$ be a Hallouin-Perret-maximal curve defined over ${\mathbb F}_q$ of genus $g$.
Then $$g \leq 23q^2 \log q.$$
\noindent Moreover, 
when we write $L_X(T)=(1+a T +qT^2)^g$
 then

\begin{enumerate}[label=(\roman*)]
\item If 
$0<a<2\sqrt{q}$
then $g<(\sqrt{q}+1)^4(\frac{q^2+1}{2q^2})$.
\item 
If $a=2\sqrt{q}$
 then $g\leq \frac{q-\sqrt{q}}{2}$.
\item 
If $a=-2\sqrt{q}$
 then $g\leq \frac{(\sqrt{q}+1)^2}{2\sqrt{q}}$.
\end{enumerate}
\end{proposition}

\begin{proof}
By Proposition \ref{Alpha}
we know that if $X$ is a HP-maximal curve 
then $f_{\Jac(X)}$ only admits one root $\omega$ and its conjugate.
With our usual conventions 
we have
 $a=-2\alpha=-(\omega+\overline{\omega})$.
The theorem of Elkies, Howe and Ritzenthaler quoted above
applies easily as $S$ is reduced to one element.
So $s=1$ and we obtain the general bound
$g \leq 23q^2 \log q$.

(i) 
When $0<a<2\sqrt{q}$
we know that the real part $\alpha$ of $\omega$ lies in the open interval $(-\sqrt{q},0)$.
So the Frobenius angle $\theta$
 which is defined as the unique 
 non-negative 
 argument of the  inverse roots of $L_X$
satisfies $\pi/2 < \theta < \pi$.
We deduce 
$\frac{1}{2} < \frac{\pi}{2\theta} < 1$,
so by definition we know that
$r=2\lceil\frac{\pi}{2\theta}\rceil=2$
and Theorem 1.1 of \cite{EHR} ensures that
$g<(\sqrt{q}+1)^4(\frac{q^2+1}{2q^2})$.

(ii) 
If $a=2\sqrt{q}$ then $\omega=-\sqrt{q}$ 
and we recognize  the case where the curve is  Weil-maximal.
The bound on $g$ is thus a consequence 
of the results of Ihara in \cite{Ihara}
(see also Theorem 5.1.1 in \cite{Serre}).

(iii) 
If $a=-2\sqrt{q}$ then $\omega=\sqrt{q}$ and
 $\theta=0$. 
 (This time the curve is  Weil-minimal.) 
 As $S$ is reduced to the set $\{0\}$,
by definition we have that $r=1/2$ and Theorem 1.1 of \cite{EHR} 
gives
$g\leq \frac{(\sqrt{q}+1)^2}{2\sqrt{q}}$.
\end{proof}

\subsection{In coverings}

Let $Y\longrightarrow X$ be a non-constant morphism of curves over ${\mathbb F}_q$. We know that if $Y$ attains the Weil upper bound (or the Weil lower bound), the same holds for $X$ (see Theorem 5.2.1. of \cite{Serre}).
More generally, it is proved in Corollary 13 of \cite{Aubry_Perret_FFA} that if $Y\longrightarrow X$ is a finite flat morphism between two varieties over a finite field, then the reciprocal polynomial of the characteristic polynomial of the Frobenius endomorphism on the $i$-th \'etale cohomology group of $X$ divides that of $Y$. In particular, if $Y\longrightarrow X$ is a non-constant morphism of curves, the $L$-polynomial $L_X(T)$ of $X$ divides the $L$-polynomial $L_Y(T)$ in ${\mathbb Z}[T]$.
We can deduce the following statement.
\begin{proposition}\label{prop:coverings}
Let $Y\longrightarrow X$ be a non-constant morphism of curves over ${\mathbb F}_q$.
If $Y$ is a Hallouin-Perret-maximal curve then $X$ is also a Hallouin-Perret-maximal curve.
\end{proposition}

\begin{proof}
If $Y$ is a HP-maximal curve, then by Proposition \ref{Alpha} the real parts of the eigenvalues of the Frobenius on $Y$ are all equal.
Since $L_X(T)$ divides $L_Y(T)$ the same holds for $X$.
\end{proof}

\section{HP-defect of a curve}
In this section we introduce the HP-defect of a curve over a finite field
as a measure of how far the curve is from being Hallouin-Perret-maximal,
that is of how far 
the inequality (\ref{HP-bound}) is from the case of equality.  
\subsection{Definition and first properties}
To define the HP-defect
we naturally consider the difference between
the right hand side and the left hand side of the inequality (\ref{HP-bound}) 
and we normalize
in order to work with integers.

\begin{definition}\label{definition:HP}
Let $X$ be a curve of genus $g\geq 1$ defined over ${\mathbb F}_q$.
The HP-defect of $X$, denoted by $\delta_{HP}(X)$, is defined by 
$$\delta_{HP}(X)=2qg^2-\left(\sharp X({\mathbb F}_q)-(q+1) \right)^2-g\left( \sharp X({\mathbb F}_{q^2}) -(q^2+1)\right).$$
\end{definition}

The following theorem aims to give different expressions of $\delta_{HP}$
in terms of the real parts $\alpha_j$'s of the 
  eigenvalues $\omega_j$'s of the Frobenius endomorphism on $X$.
It will prove useful to give alternative and elementary proofs of (\ref{HP-bound})
and to obtain other bounds on $\delta_{HP}$.

As we will provide an interpretation of $\delta_{HP}$ 
in terms of statistical tools, 
we briefly clarify the context.
When we consider a statistical sample
$z=(z_1,\ldots,z_g) \in \mathbb{R}^g$,
one naturally associates its mean $\overline{z}$
defined by $\overline{z} :=(z_1+\ldots+z_g)/g$.
The variance is another important numerical value associated to a sample 
which is defined as a measure of dispersion around the mean by the equality
$V(z):=\sum_{j=1}^g (z_j - \overline{z})^2/g$.

\begin{theorem} \label{delta}
Let $X$ be a curve of genus $g\geq 1$ defined over ${\mathbb F}_q$ and let
$\alpha_1,\ldots, \alpha_g$ be the real parts of its Frobenius eigenvalues.
Then we have:

\begin{enumerate}[label=(\roman*)]
\item \begin{equation*} \label{delta_1}
\delta_{HP}(X)=4 \left(g \sum_{j=1}^g \alpha_j^2- \left( \sum_{j=1}^g \alpha_j \right)^2 \right).
\end{equation*}

\item If we set $\displaystyle{ \sigma_1=\sum_{j=1}^g \alpha_j}$ and $\sigma_2=\displaystyle{ \sum_{1 \leq j <k \leq g} \alpha_j \alpha_k}$ 
then we have
\begin{equation*} \label{delta_3}
\delta_{HP}(X)=4 \left((g-1) \sigma_1^2 - 2g \sigma_2 \right) \textrm{ and}
\end{equation*}
\item \begin{equation*} \label{delta_2}
\delta_{HP}(X)=4 \sum_{1\leq j<k \leq g} \left( \alpha_j - \alpha_k \right)^2.
\end{equation*}
\item 
If we consider $(\alpha_j)_{1 \leq j \leq g}$ as a statistical sample
 whose mean is given by 
$\displaystyle{\overline{\alpha}:=\frac{1}{g} \sum_{j=1}^g \alpha_j}$ and  whose variance equals
 $\displaystyle{V({\alpha}):=\frac{1}{g} \sum_{j=1}^g \left(\alpha_j-\overline{\alpha} \right)^2}$,
then we have
\begin{equation*} \label{delta_4}
\delta_{HP}(X)=4g^2 V(\alpha).
\end{equation*}
\end{enumerate}
\end{theorem}

\begin{proof}
\textit{(i)}
 If we still denote by $\omega_j$'s the Frobenius eigenvalues of  $X$, we get
from the definition of the HP-defect:
$$\delta_{HP}(X)=2qg^2- \bigl(-\sum_{j=1}^g (\omega_j + {\overline \omega_j})\bigr)^2
-g \bigl(-\sum_{j=1}^g (\omega_j^2 + {\overline \omega_j}^2)\bigr)$$
which implies that
$$\delta_{HP}(X)=2qg^2- 4\bigl(
\sum_{j=1}^g \alpha_j\bigr)^2
+g \sum_{j=1}^g\bigl(4\alpha_j^2-2q\bigr)$$
and the result follows.\\
\noindent \textit{(ii) and (iii)} These points are direct consequences of point (i). \\ 
\noindent \textit{(iv)} If we factorize the equality (i) by $4g^2$ 
we get 
$\displaystyle{4g^2 
\bigg( \frac{1}{g}\sum_{j=1}^g \alpha_j^2 -  \Big( \frac{1}{g} \sum_{j=1}^g \alpha_j  \Big)^2} \bigg)$,
and we recognize
 $\overline{\alpha^2}-\overline{\alpha}^2$,
that is a classical formulation of the variance of  ${\alpha}$.
Note that the equality (iii) is also a well-known expression for the variance of ${\alpha}$
(see Section 2.4 in \cite{Mood} for example).
\end{proof}
The previous theorem gives
 alternative and elementary proofs of the Hallouin-Perret bound (\ref{HP-bound}) 
 which simply reads $\delta_{HP}(X)\geq 0$.
 It also provides an interpretation of this bound in terms of non-negativity
 of the variance of the  sample $(\alpha_j$).
From points (i) or (iv) one can also
 immediately deduce an upper bound on $\delta_{HP}(X)$ and then 
 a symmetric formulation which completes inequality (\ref{HP-bound}).
\begin{corollary}\label{intervalle}
Let $X$ be a curve of 
genus $g\geq 2$
defined over ${\mathbb F}_q$.
We have 
$$ \delta_{HP}(X) \in[0, 4qg^2]
 $$
that is
\begin{equation}\label{G2IHP}
\left|  \# X(\mathbb{F}_{q^2}) - (q^2+1) + \frac{1}{g} \left( \# X(\mathbb{F}_q) -(q+1) \right)^2 \right|  \leq 2qg.
\end{equation}
\end{corollary}

Let us exploit the variance formulation of $\delta_{HP}(X)$.
It is a classical statistical problem to give bounds 
for the variance of a sample $(\alpha_j)_{1 \leq j \leq g}$ 
of elements in a range $\left[ \alpha_{\textrm{min}};\alpha_{\textrm{max}} \right]$.
In \cite{KaiSpan} (see Lemma 1)
Kaiblinger and Spangl 
state that we always have 
$$
V(\alpha) \leq \frac{1}{4} (\alpha_{\textrm{max}}-\alpha_{\textrm{min}})^2,$$
and that,
if $g$ is odd 
the bound  
can be improved this way
$$V(\alpha) \leq \frac{1}{4} \left( 1-\frac{1}{g^2}  \right) (\alpha_{\textrm{max}}-\alpha_{\textrm{min}})^2.$$
Lemma 1 in \cite{KaiSpan}
 also indicates that when $g$ is even this upper bound
 on $V(\alpha)$
  is reached
if and only if half of the values $\alpha_j$ equal $- \sqrt{q}$
whereas the other half equal $ \sqrt{q}$.
When $g$ is odd, the upper bound is reached if and only if 
$(g-1)/2$ or $(g+1)/2$  values equal $- \sqrt{q}$ whereas the others equal $\sqrt{q}$.
We deduce the following proposition.
 \begin{proposition} \label{improve_HP}
Let $X$ be a curve of 
genus $g\geq 2$ 
defined over ${\mathbb F}_q$
and let
$\alpha_1,\ldots, \alpha_g$ be the real parts of its Frobenius eigenvalues.
\begin{enumerate}[label=(\roman*)]
\item If $g$ is even then
\begin{equation} \label{GIHPpair}
-2qg \leq  \# X(\mathbb{F}_{q^2}) - (q^2+1) + \frac{1}{g} \left( \# X(\mathbb{F}_q) -(q+1) \right)^2
\end{equation} 
and equality holds if and only if $g/2$ values $\alpha_j$ equal $-\sqrt{q}$
whereas the others equal $\sqrt{q}$.
\item  If $g$ is odd then
\begin{equation} \label{GIHPimpair}
-2q \left(g-\frac{2}{g} \right) \leq  \# X(\mathbb{F}_{q^2}) - (q^2+1) + \frac{1}{g} \left( \# X(\mathbb{F}_q) -(q+1) \right)^2 
\end{equation} 
and equality holds if and only if $(g-1)/2$ or $(g+1)/2$ values  $\alpha_j$ equal $-\sqrt{q}$
whereas the others equal $\sqrt{q}$. 
\end{enumerate}
\end{proposition}

Actually the inequality (\ref{GIHPpair})
is less accurate than the Weil bound applied to the curve 
considered over $\mathbb{F}_{q^2}$, 
but we indicate it 
for the sake of completeness.
However, the bound (\ref{GIHPimpair}) can sometimes yield valuable information.
For example, consider a genus $3$ curve defined over $\mathbb{F}_4$ such that $N_1=5$.
The Weil bound (\ref{Weilbound}) yields $N_2 \geq 5$
whereas (\ref{GIHPimpair}) enables to deduce from the extra information 
about $N_1$
that $N_2 \geq 9$.
Actually the LMFDB database \cite{LMFDB}
provides the existence of such curves\footnote{\url{https://www.lmfdb.org/Variety/Abelian/Fq/3/4/a_ad_ac}} 
with  $N_2=11$
(for example the curve of equation
$a x^4+x y^3+x z^3+y^4=0$
where $a \in \mathbb{F}_{4}$ satisfies 
$a^2+a+1=0$).
Moreover it indicates the existence 
of an abelian variety\footnote{\url{https://www.lmfdb.org/Variety/Abelian/Fq/3/4/a_ae_a}}   
whose associated virtual curve
would have $9$ rational points over $\mathbb{F}_{16}$.
(For reference, this abelian variety is defined over $\mathbb{F}_4$,
has dimension $3$,
and its $L$-polynomial is given by
$ (1-2T)^2(1+2T)^2(1+4T^2)$.)

\begin{remark}\label{remark:totof}
The necessary conditions on the $\alpha_j$'s established
in  point (i) of Proposition \ref{improve_HP} help us to find,
with the help of the  LMFDB database \cite{LMFDB},
curves of genus $2$
which reach the lower bound
$-2qg$.
Let us give some examples when $g$ is even.
For instance 
the (Diophantine-stable) hyperelliptic curve $X$ of equation $y^2=x^5+4x$ 
defined over $\F_5$ satisfies $\#X(\F_q)=\#X(\F_{q^2})=6$ 
and we can check the equality
$\# X(\mathbb{F}_{q^2}) - (q^2+1)+\frac{1}{g}\left( \# X(\mathbb{F}_q) -(q+1) \right)^2 =-2qg$.
We have also found examples when $q=7$, $q=8$ or $q=11$.
For the case $q=9$ the necessary conditions lead us to
$N_1=10$, $N_2=46$, and to an isogeny class which does not contain any Jacobian
according to the LMFDB database\footnote{\url{https://www.lmfdb.org/Variety/Abelian/Fq/2/9/a\_as}}.
(For reference, this isogeny class is associated to the $L$-polynomial 
$ (1-3T)^2(1+3T)^2$.)
For the case $q=13$ one can find (at least) two hyperelliptic curves of genus $2$ 
with equations $y^2=x^5+12x$ and
$y^2=x^6+2x^3+8$ such that
$\alpha_1=-\sqrt{13}$ and $\alpha_2=\sqrt{13}$.
We thus have
$N_1=14$, $N_2=118$ and the lower bound is reached. 
 \end{remark}

\begin{remark}\label{remark:christophe}
We are indebted to Christophe Ritzenthaler for pointing out 
the following example of a non-elliptic curve of odd genus which reaches the lower bound
$-2q(g-\frac{2}{g})$.
We take $q=49$ and we consider a generator $a$ of the multiplicative group $\mathbb{F}_q^{\ast}$.
Thus the curve of equation 
$x^{4} + a^{43} x^{3} y + a^{36} x^{3} z + a^{27} x^{2} y^{2} + a^{10} x^{2} y z 
+ a^{31} x^{2} z^{2} + a^{9} x y^{3} + a^{47} x y^{2} z + a^{6} x y z^{2} + a^{19} x z^{3} 
+ a^{9} y^{4} + 6 y^{3} z + a^{46} y^{2} z^{2} + a^{22} y z^{3} + 6 z^{4} = 0$
has genus $3$ and is such that $N_1=2108$ whereas $N_2=36$,
and so we can check
$\# X(\mathbb{F}_{q^2}) - (q^2+1)+\frac{1}{g}\left( \# X(\mathbb{F}_q) -(q+1) \right)^2 =-2q(g-\frac{2}{g})$.
The curve is obtained by twisting the curve of equation $x^4+y^4+z^4=0$
so that it only changes one elliptic factor.
\end{remark}

\subsection{Weil polynomials and Hallouin-Perret-maximal curves}\label{subsection:qWeil}
In this subsection
 we relate the topic of HP-maximal curves
 with the question of determining whether
a polynomial 
$f(T)=1+a_1T+a_2T^2+\cdots+a_2 q^{g-2} T^{2g-2} + a_1 q^{g-1} T^{2g-1}+q^g T^{2g} \in {\mathbb Z}[T]$
can be the reciprocal polynomial 
of the characteristic polynomial 
of a dimension $g$
abelian variety over ${\mathbb F}_q$. 
We provide necessary conditions in terms of the two first coefficients $a_1$ and $a_2$ 
and in terms of $\delta_{HP}$.

Recall that a $q$-Weil number is an algebraic integer such that its image under every embedding 
has absolute value $\sqrt{q}$ and
that a monic polynomial of ${\mathbb Z}[T]$ is called a $q$-Weil polynomial 
if all its roots are $q$-Weil numbers.
The Honda-Tate theorem (see \cite{Honda-Tate})  
establishes a bijection between the simple abelian varieties over ${\mathbb F}_q$
up to isogeny 
and the $q$-Weil numbers up to conjugation
so that a well-known necessary condition for $f$ is to be a $q$-Weil polynomial.

When  $g=2$
R\"uck has proved  (Theorem 1.1 in \cite{Ruck})
that a necessary condition 
for $f$ to be a $q$-Weil polynomial
is that
$a_2 \leq \frac{a_1^2}{4}+2q$.
In the same direction, 
Haloui has proved  (Theorem 1.1 in \cite{Haloui_JNT})
that for $g=3$ 
a necessary condition is that
$a_2 \leq \frac{a_1^2}{3}+3q$. 
Haloui and Singh have also proved  (Theorem 1.1 in  \cite{Haloui-Singh})
that if
$g=4$
a necessary condition is given by
$a_2 \leq \frac{3}{8}a_1^2+4q$.
Lastly, when  $g=5$ Sohn has stated
 (Theorem 1.2  in \cite{Sohn})
 that the condition becomes
$a_2 \leq \frac{2}{5}a_1^2+5q$.
We will propose a generalization
 of this series of necessary
  conditions for any $g$,
  as well as a lower bound for the coefficient $a_2$.

We start expliciting relations between
 the coefficients $a_1, a_2$ and the defect $\delta_{HP}(X)$
 in the case where $f(T)$ is 
the $L$-polynomial of a curve $X$. 
\begin{proposition}\label{somme_des_a_i}
Let $X$ be a curve defined over ${\mathbb F}_q$ of genus $g\geq 2$
 and with HP-defect $\delta_{HP}$.
For any $j \in \{1,\ldots,g \}$ we also note $\alpha_j$ for the real part of the inverse root $\omega_j$ 
of the $L$-polynomial
$L_X(T)=1+a_1T+a_2T^2+\cdots+a_2 q^{g-2} T^{2g-2} + a_1 q^{g-1} T^{2g-1}+q^g T^{2g}.$
We have
\begin{equation}\label{equation:somme_alpha_i}
\sum_{j=1}^g \alpha_j^2=\frac{a_1^2+\delta_{HP}}{4g}.
\end{equation}
We deduce
\begin{equation}\label{equation:inegalite_a2}
a_2 \leq \frac{a_1^2(g-1)}{2g} +gq
\end{equation}
with equality if and only if $X$ is a Hallouin-Perret-maximal curve. 
We also have the inequality
\begin{equation}\label{equation:lower_bound_a_2}
a_2 \geq 
\begin{cases}
\frac{g-1}{2g} a_1^2 -gq  & \text{when } g \geq 2 \text{ is even}, \\
\frac{g-1}{2g} a_1^2 +\frac{(2-g^2)}{g}q & \text{when } g  \geq 3 \text{ is odd}.
\end{cases}
\end{equation}
\end{proposition}
\begin{proof}
For the first point  
we express equality (\ref{delta_1})
in Theorem \ref{delta}
in terms of the coefficient $a_1$ 
which satisfies $a_1=-2 \sum_{j=1}^g \alpha_j$.

For the second point we start from equality (\ref{delta_3}) in Theorem \ref{delta}.
We use $a_1=-2 \sigma_1$ and $a_2=gq+4\sigma_2$
to get  
\begin{equation}\label{equation:delta_HP_a1_a2}
2ga_2+\delta_{HP}=(g-1)a_1^2+2g^2q
\end{equation}
and thus
$a_2=a_1^2(g-1)/2g+gq-\delta_{HP}/2g$.
But $\delta_{HP}$ is non-negative
and equals zero if and only if the curve is HP-maximal.

For the third point we just exploit the inequalities
$\delta_{HP} \leq 4 g^2 q$ when $g$ is even
and 
$\delta_{HP} \leq 4 (g^2-1) q$ when $g$ is odd 
together with (\ref{equation:delta_HP_a1_a2}).
\end{proof}

\begin{remark}
Let us point out a geometric interpretation in the euclidean space ${\mathbb R}^g$
if we associate to a curve $X$ the point $P_X$ of coordinates  $(\alpha_j)_{1\leq j \leq g}$.
Thanks to the Riemann Hypothesis we know that
this point belongs to the closed ball 
$\overline{B}_{\infty}(0,\sqrt{q})$.
Now we fix the value of $a_1$.
A way to translate Proposition (\ref{somme_des_a_i}) is to say that 
the point $P_X$ belongs to the affine plane $\mathcal{P}$ of equation $ \sum_{j=1}^g x_j=-\frac{a_1}{2}$ 
as well as to the sphere  $\mathcal{S}$ of equation $\sum_{j=1}^g x_j^2=\frac{a_1^2+\delta}{4g}$.
Since the radius 
$r:=\sqrt{\frac{a_1^2+\delta_{HP}}{4g}}$
of $\mathcal S$  
is greater than or equal to the distance $d:=\frac{|a_1|}{2\sqrt{g}}$
 from the origin to the plane $\mathcal P$, 
 we deduce that the intersection of $\mathcal S$ and $\mathcal P$ is nonempty.
Moreover, $d=r$ if and only if $\delta_{HP}=0$.
In other words, $X$ is a HP-maximal curve if and only if $\mathcal{P}$ is tangent to $\mathcal{S}$.
\end{remark}

Until the end of this section,
we no longer consider any algebraic curve $X$.
We just assume that $f$
is a $q$-Weil polynomial
and we still denote by $\alpha_j$
the real parts of its complex roots $\omega_j$. 
We still consider
$\tilde{\delta}:=4g^2V(\alpha)$
where  $V(\alpha)$ is the variance of the $\alpha_j$'s.
The real number $\tilde{\delta}$ no longer expresses 
in terms of the number of points of a curve over a finite field
and over its quadratic extension, 
but the different results of Theorem \ref{delta}
(with $\tilde{\delta}$ in place of $\delta_{HP}$) 
as well as 
the inequalities (\ref{equation:inegalite_a2}) and (\ref{equation:lower_bound_a_2}) remain valid.
We thus obtain the following generalization
of the results of 
R\"uck,
Haloui,
Haloui and Singh, and Sohn
together with a lower bound on the coefficient $a_2$.

\begin{theorem}\label{th:bound_a_2}
We consider $g \geq 2$.
If
 $T^{2g}+a_1T^{2g-1}+a_2T^{2g-2}+\cdots+a_2 q^{g-2} T^{2} + a_1 q^{g-1} T+q^g $
is  a $q$-Weil polynomial then
\begin{equation*}
a_2 \leq \frac{a_1^2(g-1)}{2g} +gq
\end{equation*}
and 
\begin{equation}\label{equation:lower_bound_a2}
a_2 \geq
\begin{cases}
\frac{g-1}{2g} a_1^2 -gq  & \text{ when } g\geq 2 \text{ is even}, \\
\frac{g-1}{2g} a_1^2 +\frac{(2-g^2)}{g}q & \text{ when } g  \geq 3 \text{ is odd}.
\end{cases}
\end{equation}
\end{theorem}

\begin{remark}
For instance let us consider the case where $g=3$ 
and  compare with the lower bound
$a_2 \geq 4 \sqrt{q} \left| a_1 \right| -9q$
 provided in \cite{Haloui_JNT}.
The bound (\ref{equation:lower_bound_a2}) reads
$ a_2 \geq a_1^2/3 - \frac{7}{3}q$
and 
thus
is better if and only if
$a_1^2/3 - 7q/3 \geq 4 \sqrt{q} \left| a_1 \right| -9q$.
Taking into account that $\left| a_1 \right| \leq 6 \sqrt{q}$ 
(see Condition (1) in Theorem 1.1 in \cite{Haloui_JNT})
the condition becomes
$\left| a_1 \right| \leq 2 \sqrt{q}$.
This may occur as 
the 
example\footnote{\url{https://www.lmfdb.org/Variety/Abelian/Fq/3/4/a\_ae\_a}}
 already given above
of 
an isogeny class of an abelian variety 
defined over $\mathbb{F}_4$,
of dimension $3$ and 
whose $L$-polynomial is given by
$ (1-2T)^2(1+2T)^2(1+4T^2)$
leads to 
$a_1=0$ and $a_2=-4$.
In this case our bound (\ref{equation:lower_bound_a2})  yields $a_2 \geq -28/3$
whereas the bound $a_2 \geq 4 \sqrt{q} \left| a_1 \right| -9q$ in \cite{Haloui_JNT}
gives $a_2 \geq -36$.

To be fair, we stress that 
(\ref{equation:lower_bound_a2}) 
only provides a lower bound on $a_2$,
whereas the sets of inequalities given in 
Theorem 1.1 in \cite{Ruck}, 
Theorem 1.1 in \cite{Haloui_JNT},
Theorem 1.1 in \cite{Haloui-Singh},
Theorem 1.2 in \cite{Sohn}
and studied in depth by Marseglia in \cite{Marseglia}
completely characterize the $q$-Weil polynomials of small degrees.

\end{remark}
\section{Ihara-maximal curves}\label{section:Ihara}

This section is devoted to the study
of curves whose number of rational points
reaches the Ihara bound (\ref{Iharabound}), which 
 are
called Ihara-maximal curves.

\subsection{A characterization of Ihara-maximal curves}

As stated in Proposition \ref{Ihara-max} below,
the Ihara-maximal curves will appear
as the curves which are both Hallouin-Perret-maximal
and Diophantine-stable.

\begin{proposition}\label{Ihara-max}
Let $X$ be a curve of
 genus $g\geq 1$ 
defined over ${\mathbb F}_q$.
The following assertions are equivalent.
\begin{enumerate}[label=(\roman*)]
\item $X$ is a Ihara-maximal curve.
\item$X$ is both a Hallouin-Perret-maximal curve and a Diophantine-stable curve with respect to the extension 
${\mathbb F}_{q^2} / {\mathbb F}_q$.
\item $Z_X(T)=\frac{(1-2\alpha T +qT^2)^g}{(1-T)(1-qT)}$
 where $\alpha=\frac{1}{4}- \frac{\sqrt{ (8q+1)g^2+ (4q^2-4q)g}}{4g}$.
\end{enumerate}
\end{proposition}

\begin{proof}
We still write $\alpha_1,\ldots,\alpha_g$ for the  real parts of the Frobenius eigenvalues of $X$.
For the purpose of proving the first implication $(i) \Rightarrow (ii)$,
we recall that
the original proof of Ihara (\textsection 2 in \cite{Ihara})
rests on three inequalities.
More precisely, 
the Cauchy-Schwarz inequality
$g\sum_{j=1}^g\alpha_j^2\geq \bigl(\sum_{j=1}^g \alpha_j\bigr)^2$
together with
the arithmetic inequality $N_1 \leq N_2$ which reads
\begin{equation}\nonumber
N_1 \leq q^2+1+2qg -4\sum_{j=1}^g \alpha^2_j
\end{equation}
lead to the following quadratic inequality in $N_1$
\begin{equation}\label{equation:quadratic}
N_1^2-(2q+2-g)N_1 +(q+1)^2-(q^2+1)g-2qg^2 \leq 0.
\end{equation}
When $X$ is Ihara-maximal, the three inequalities become equalities.
But the equality $N_1=N_2$ defines the Diophantine-stability.
And the Cauchy-Schwarz inequality becomes an equality if and only if
 the $\alpha_j$'s are all equal,
 which characterizes HP-maximal curves by Proposition \ref{Alpha}.
 
To prove 
that
$(ii) \Rightarrow (iii)$
one can notice that
$N_2 - (q^2+1) = 2gq-\frac{1}{g}\bigl(N_1 - (q+1)\bigr)^2$
and
$N_1=N_2$ imply 
that inequality (\ref{equation:quadratic}) 
becomes an equality.
The only non-negative solution is
$N_1=  q+1+\frac{1}{2}\left(\sqrt{ (8q+1)g^2+ (4q^2-4q)g} -g\right)$
and thus one can use Proposition \ref{Alpha}
to obtain the claimed form for the zeta function of $X$.

Finally, if $(iii)$ is verified 
we are again in the context of Proposition  \ref{Alpha}
and then the knowledge of the value 
$\alpha=\frac{q+1-\sharp X({\mathbb F}_q)}{2g}$ enables
to determine $N_1$ and to conclude that $X$ is Ihara-maximal.
\end{proof}

\begin{remark}
As a straightforward application we identify a family 
of Ihara-maximal curves.
Indeed, 
for $t \geq 1$ and 
$q=2^{2t+1}$ 
let us consider the Deligne-Lusztig curve of Suzuki type (called Suzuki curve from now), 
associated to the Suzuki group
$Sz(q)$, that is the
curve defined over ${\mathbb F}_q$
as the non-singular model $S$
of the plane curve given by the equation
$y^q-y=x^{q_0}(x^q-x)$
where $q_0=2^t$. 
It is well-known 
(see \cite{H-S}, \cite{HKT_book} or Proposition 4.3 of \cite{Hansen})
 that this curve has genus
$g=\frac{\sqrt{q}(q-1)}{\sqrt{2}}$,
satisfies $L_{ S}(T)=(1+ 2q_0 T + qT^2)^g$
and has $q^2+1$ rational points over ${\mathbb F}_{q}$
and over  ${\mathbb F}_{q^2}$.

So this curve is Diophantine-stable, and Proposition (\ref{Alpha})
ensures the Hallouin-Perret-maximality. 
A Suzuki curve is thus Ihara-maximal.
\end{remark}

\subsection{An analog of a theorem of R\"uck and Stichtenoth}\label{subsection:analog}

Ihara has proved in \cite{Ihara} that a Weil-maximal curve $X$ defined over ${\mathbb F}_q$ 
 has a genus
  less than or equal  to $\frac{\sqrt{q}(\sqrt{q}-1)}{2}$.
 Indeed, 
 the Ihara bound $N_2^*:=q+1+(\sqrt{(8q+1)g^2+4qg(q-1)}-g)/2 $ 
 becomes sharper 
 than the Weil-bound 
 $N_1^*:=q+1+2g\sqrt{q} $
 for $g > g_2:=\frac{\sqrt{q}(\sqrt{q}-1)}{2}$. 
 Hallouin and Perret have proposed an even sharper bound
 $N_3^*$ (see \cite{Hallouin-Perret} and \cite{A-I}) valid when $g \geq g_3:=\frac{\sqrt{q}(q-1)}{\sqrt{2}}$,
which implies that the genus of a Ihara-maximal curve is less than or equal to $g_3$.

And in this direction they provide
(Theorem 14 of \cite{Hallouin-Perret})
an increasing sequence $(g_n)_{n \geq 1}$ of integers
and a sequence of upper bounds $N^*_n(g)$
such that $N^*_n(g)$ is a valid 
bound for $\#X(\mathbb{F}_q)$ when the genus $g$ 
of the curve is greater than or equal to $g_n$.
The bounds are proven to be sharper and sharper 
(see point 18 of Theorem 14 in \cite{Hallouin-Perret})
and the following expression
for $g_n$ is 
established:
 $
 \displaystyle{g_n=\sqrt{q}^{n+1} \sum_{k=1}^{n}\frac{1}{\sqrt{q}^k}\cos\left( \frac{k\pi}{n+1}\right)}.$

Let us now come back to the case of the Weil bound.
R\" uck and Stichtenoth have characterized the Weil-maximal curves of maximal genus
the following way.

 \begin{theorem}\label{th:RS} \textbf{(R\" uck and Stichtenoth, main theorem in \cite{R-S})}
We suppose that $q$ is a square
and we consider a curve $X$ defined over ${\mathbb F}_q$.
Suppose that $X$ has genus $g=\frac{\sqrt{q}(\sqrt{q}-1)}{2}$.
Then $X$ is Weil-maximal if and only if $X$ is ${\mathbb F}_q$-isomorphic to the 
Hermitian curve  whose equation is
$y^{\sqrt{q}}+y=x^{\sqrt{q}+1}.$
\end{theorem}

We notice that it is possible to obtain
an analogue  of this result 
for Ihara-maximal curves
by reformulating a maximality theorem of 
Fuhrmann and Torres (Theorem 2 in \cite{F-T})
for Suzuki curves. This theorem asserts that if
$t\geq 1$, $q=2^{2t+1}$
 and $q_0=2^t$ then
any curve 
of genus $g=q_0(q-1)$ and
 such that $\#X(\mathbb{F}_q)=q^2+1$ is isomorphic 
 to the Suzuki curve.
 It is thus sufficient to 
 verify (with a tedious but straightforward computation)
 that in this setting
 $q^2+1$ equals the Ihara-bound 
 $ q+1+(\sqrt{(8q+1)g^2+4qg(q-1)}-g)/2$ 
 to obtain the following
analogue of the R\" uck and Stichtenoth theorem.

\begin{theorem}\label{theorem:suzuki_maximal}
\textbf{(Reformulation of Theorem 2 in \cite{F-T})}
We consider $t\geq 1$ and $q=2^{2t+1}$. 
Let $X$ be a curve defined over ${\mathbb F}_q$.
Suppose that $X$ has
 genus $g=\frac{\sqrt{q}(q-1)}{\sqrt{2}}$.
Then
$X$ is Ihara-maximal 
if and only if $X$ is ${\mathbb F}_q$-isomorphic to the Suzuki curve $S$
which is the non-singular model of the curve of
 equation 
$y^q-y=x^{q_0}(x^q-x)$
where $q_0=2^t$.
\end{theorem}

\subsection{Determination of Ihara-maximal curves 
for small values of $g$ or $q$}

\begin{proposition}\label{th:list_Ihara_g}
We give in the  Table \ref{table:ihara_g} the complete list, up to isomorphism, 
of Ihara-maximal curves defined 
over ${\mathbb F}_q$ of genus $g\leq 18$,
except for the case $g=7$.
In this case, 
we know that $q=7$ and
that there exists at least one Ihara-maximal curve,
but we do not know if there is unicity.
\end{proposition}

\begin{center}
\begin{table}[h]
\footnotesize
\centering
\begin{tabular}{|c|c|c|c|c|c|}
\hline
$g$ & $q$ & $N_1=N_2$ & $L_X(T)$ &  Ihara-maximal curve & Weil-max   \\ 
\hline
\hline
 & $2$ & $5$ & $1+2T+2T^2$ &  $y^2+y=x^3+x$ &   \\
& & & & & \\
$1$ & $3$ & $7$ & $1+3T+3T^2$ & $y^2=x^3+2x+1$  &   \\
& & & & & \\
 & $4$ & $9$ & $1+4T+4T^2$ & $x^3+y^3+z^3=0$  & $\times$   \\
\hline
 & $4$ & $14$ & $(1+3T+4T^2)^3$ &  $x^4+x^2y+xy^3+x+y^2=0$ &   \\
& & & & & \\
& $4$ & $14$ & $(1+3T+4T^2)^3$ & $x^4+x^2y^2+y^4+x^2y+xy^2+$  &  \\
$3$ & & & & $+x^2+xy+y^2+1=0$  &\\
& & & & & \\
 & $9$ & $28$ & $(1+3T)^6$ & $x^4+y^4+z^4=0$  & $\times$   \\
\hline
$6$ & $16$  & $65$   & $(1+8T+16T^2)^6$ & $x^5+y^5+z^5=0$ & $\times$  \\
\hline
$7$ & $7$ & $36$ & $(1+4T+7T^2)^7$ & There exists at least one curve: & \\
& & & &the fibre product &  \\
& & & & $y_1^3=5(x+2)(x+5)/x$ &  \\
& & & & $y_2^3=3x^2(x+5)/(x+3)$ &  \\
\hline
$10$ & $25$ & $126$ & $(1+10T+25T^2)^{10}$  & $x^6+y^6+z^6=0$ & $\times$  \\
\hline
$14$ & $8$ & $65$ & $(1+4T+8T^2)^{14}$  &  $y^8-y=x^2(x^8-x)$ &  \\
\hline
\end{tabular}
\caption{
\footnotesize{List (up to isomorphism) of all  Ihara-maximal curves  of genus $g \leq 18$ except for $g=7$.}}
\label{table:ihara_g}
\end{table}
\end{center}

We also provide in each case
the number $N_1=N_2$ of rational points over $\mathbb{F}_q$ 
(and $\mathbb{F}_{q^2}$),
the $L$-polynomial of the curve,
an equation of an affine model of the curve 
and we indicate with a $\times$ in the sixth column
whether the curve is  Weil-maximal.

\begin{proof}
We notice that if $g < (q-\sqrt{q})/2$
then the Weil bound $N_1^*$ (with the notations of Subsection \ref{subsection:analog})
is sharper than the Ihara bound $N_2^*$,
and so a curve cannot be Ihara-maximal.
It enables to discard these cases.
Thus, with the help of a Python program
 we list all the couples $(g,q)$ for which
 $(q-\sqrt{q})/2 \leq g \leq 18$ and for which $(8q+1)g^2+4qg(q-1)$ is a square. 
 We also list the corresponding values of $N_1(=N_2)$.

Some of these couples can be discarded because
the Ihara bound is greater than a known upper bound 
 of 
the maximal number
$N_q(g)$
of rational points of a curve of genus $g$ over ${\mathbb F}_q$.
We sum up in 
Table \ref{table:discard}
these discarded couples with
the helpful references (obtained for the most part thanks to ManyPoints \cite{manypoints}).
 
\begin{center}
\begin{table}[h]
\footnotesize
\begin{tabular}{|l|c|c|c|c|c|c|c|c|c|c|c|c|}
\hline
$g$ & 4 & 6 & 8 & 8 & 8 & 10 & 10 & 15 & 16 & 16 & 18 & 18 \\
\hline
$q$ & 8 & 9 & 11 & 11 & 19 & 5 & 16 & 25 & 4 & 13 & 29 & 41 \\
\hline
Ihara bound & 29 & 40 & 56 & 29 & 88 & 36 & 87 & 161 & 45 & 102 & 204 & 270 \\ 
\hline
Up. bound  & 25 & 38 & 55 & 88 & 84 & 33 & 86 & 160 & 38 & 101 & 203 & 258 \\
on $N_q(g)$ & & & & & & & & & & & & \\
\hline
Reference & \cite{Savitt} & \cite{Howe-Lauter-EMS} & \cite{kohnlein} & \cite{kohnlein} & \cite{Serre-CRAS} 
& \cite{Howe-Lauter-EMS} & \cite{Howe-Lauter-EMS} & \cite{Howe-Lauter-EMS} & \cite{Serre-CRAS} & \cite{kohnlein} & \cite{Howe-Lauter-EMS} & \cite{Serre} \\
\hline
\end{tabular}
\caption{
\footnotesize{List of discarded cases in proof of Proposition \ref{th:list_Ihara_g}}}
\label{table:discard}
\normalsize
\end{table}
\end{center}
 Let us now treat the remaining cases with the help of 
 Proposition \ref{Ihara-max}
 which ensures that
 a curve is Ihara-maximal
 if and only if it is both HP-maximal and Diophantine-stable.
 The case of curves of genus one is easily handled
 as such a curve is always HP-maximal 
 (see Remark \ref{remark:elliptic}) and as 
 Bars, Lario and Vrioni provide 
 in Proposition 3.1 in \cite{BLV}
 the only isomorphism classes of Diophantine-stable
 elliptic curves.

 When $g=3$ and $q=4$,
 we should have $N_1=N_2=14$. In \cite{BLV} the authors indicate
 the two only isomorphism classes of Diophantine-stable curves
 for the extension $\mathbb{F}_{16} / \mathbb{F}_{4}$.
 
 When $(g,q)= (3,9),(6,16)$ or $(10,25)$ 
 we notice that $g=\sqrt{q}(\sqrt{q}-1)/2$ and 
 that the candidate curve must be Weil-maximal.
So by the Theorem of R\"uck and Stichtenoth
 quoted in Subsection \ref{subsection:analog}
 we know that we deal with an Hermitian curve.
 
 If $g=14$ and $q=8$, it is remarkable that 
 $g=\sqrt{q}(q-1)/\sqrt{2}$, so Theorem \ref{theorem:suzuki_maximal} ensures
 that in this case a Ihara-maximal curve is a Suzuki curve.
 
 Finally, \"Ozbudak, Tem\"ur and Yayla
 provide in \cite{OTY2013} 
 a fibre product of Kummer extensions
 which is 
 an example of Ihara-maximal curve
 of genus $7$ defined over $\mathbb{F}_7$ and such that $N_1=36$.
 \end{proof}

\begin{remark}\label{Ihara_max_q_petit}
We can proceed in the same way
to obtain the  list
of the Ihara-maximal curves defined over $\mathbb{F}_q$
for $q \leq 13$.
The point is to bound the possible values of $g$ for a given value of $q$.
As in the proof of Proposition \ref{th:list_Ihara_g}
we know that $(q-\sqrt{q})/2 \leq g$.
If we refer to the discussion of Subsection \ref{subsection:analog}
and to Theorem 14 in \cite{Hallouin-Perret},
we know that a curve 
of genus $g > \sqrt{q}(q-1)/\sqrt{2}$
cannot be Ihara-maximal, so $g \leq \sqrt{q}(q-1)/\sqrt{2}$.
So this time
we list all the couples $(g,q)$ for which
 $(q-\sqrt{q})/2 \leq g  \leq \sqrt{q}(q-1)/\sqrt{2}$,
 $q \leq 13$ 
  and for which $(8q+1)g^2+4qg(q-1)$ is a square. 
In comparison with the proof of Proposition \ref{th:list_Ihara_g}
the only additional case is that of
 the couple 
$(g,q)=(25,13)$ which would lead to a curve with $144$
rational points. 
But it is known that in this context $N_1 \leq 142$ (see \cite{kohnlein}). 

Thus the complete list, up to isomorphism, 
of Ihara-maximal curves defined 
over ${\mathbb F}_q$ for $q \leq 13$
of genus $g \geq 1$,
except for the case $q=7$, is given by the eight curves of Table \ref{table:ihara_g} corresponding to
the couples $(g,q)\in\{(1,2), (1,3), (1,4), (3,4), (3,9), (7,7), (14,8)\}$.

\end{remark}

\begin{remark}
So we have managed to discuss the existence
of  Ihara-maximal curves of genus $g$ 
for any $g \leq 18$ and any value of $q$.
In this direction, 
the first case we do not know how to treat 
is the one of a possible curve of genus 
$g=19$. 
With the method described above, 
we know that this curve shall be defined
over $\mathbb{F}_{19}$ and such that $N_1=N_2=153$,
 but we are unable to determine whether there exists such a curve.
 
 In the other direction,
 we know if there exists a Ihara-maximal curve
 defined over $\mathbb{F}_q$ 
 for any $q \leq 13 $. 
 To make further progress
 we would need to know if there exists
 a curve of genus
 $24$ defined over $\mathbb{F}_{16}$
with $N_1=N_2=161$.
 \end{remark}

\section{Abelian surfaces isogenous to Jacobians of Hallouin-Perret-maximal curves}\label{section:structure}

We focus in this section on 
abelian surfaces 
which are isogenous to Jacobians of 
genus 2
HP-maximal curves.
Recall that by Proposition \ref{Alpha}, the $L$-polynomial of
such a curve
reads $(1+aT+qT^2)^2$, 
and so the characteristic polynomial of the Jacobian of such a curve
expresses $(T^2+aT+q)^2$.
The next result
 will answer the natural question:
among the polynomials $(T^2+aT+q)^2$,
which ones do correspond to the
characteristic polynomial of the  Jacobian of a HP-maximal curve?

The proof 
will be based upon 
successive works which aim to describe
the characteristic polynomials 
which can be associated to
abelian and Jacobian surfaces. 
Maisner and Nart have characterized in \cite{MNH} 
when a $q$-Weil polynomial of degree 4 corresponds to an abelian surface defined over ${\mathbb F}_q$ 
and when this abelian surface is simple.
Furthermore, Howe, Nart and Ritzenthaler have determined in \cite{HNR} 
when simple abelian surfaces are isogenous to a Jacobian
and also when the isogeny class of a square of an ordinary elliptic curve contains a Jacobian.
Finally, 
we will make use of
the characterization of abelian surfaces
 whose isogeny class contains 
 the product of two supersingular elliptic curves given
 in characteristic 2 by Maisner and Nart (see \cite{M-N}),
 in characteristic 3 by Howe (see \cite{Howe})
 and in greater characteristic  by Howe, Nart and Ritzenthaler  (see \cite{HNR}).
When the abelian surface is not simple, 
we use the characterization of Waterhouse (see \cite{Waterhouse}) of a  polynomial of degree 2 arising as the characteristic polynomial of an elliptic curve.
Putting all of this together,
one obtains
 the following characterization of HP-maximal curves of genus 2
  related to
  the structure of their Jacobians.

\begin{theorem}\label{HP_curves_of_genus_2}
Let $q=p^n$ be a power of a prime $p$.
We consider a polynomial $f(T)=(T^2+aT+q)^2 \in{\mathbb Z}[T]$
(which amounts to saying that $a\in{\mathbb Z}$).
Then $f(T)$ is the characteristic polynomial of a Jacobian of a Hallouin-Perret-maximal curve defined over ${\mathbb F}_q$ if and only if
one of the following conditions holds. \\

\medskip
\begin{small}
\noindent
 {\bf 1) Simple abelian surface case}
\begin{enumerate}
\item[(1.1)] $n$ is even and $p\equiv 1\pmod 4$ and $a=0$.
\item[(1.2)] $n$ is even and $p\equiv 1\pmod 3$  and  $a=\pm \sqrt{p^n}$. 
\end{enumerate}
In these cases, $f(T)$ is the characteristic polynomial of a simple abelian surface defined over ${\mathbb F}_q$ which is supersingular and isogenous to the Jacobian of a Hallouin-Perret-maximal curve.

\bigskip

\noindent
{\bf 2) Split ordinary case}  \\ \ \\
$\vert a \vert \leq 2 \sqrt{q}, $
$(a,p)=1$ and $a^2-4q \notin \{-3, -4, -7 \}.$ \\

In this case, $f(T)$ is the characteristic polynomial of an abelian surface defined over ${\mathbb F}_q$ which is isogenous to 
$E \times E$ where $E$ is an ordinary elliptic curve and its isogeny class contains the Jacobian of a Hallouin-Perret-maximal curve.

\bigskip
 
\noindent
{\bf 3) Split supersingular case} 

\medskip

\begin{enumerate}
\item[(3.1)]   $p =2$  and  $n > 1$:

\begin{enumerate}
\item[(i)] $n$ odd and  $a=0$.
\item[(ii)] $n$ odd and $a=\pm  \sqrt{2p^n}$.
\item[(iii)] $n$ even and $a=0$.
\item[(iv)] $n$ even and $a=\pm p^{\frac{n}{2}}$.
\item[(v)] $n \geq 4$ even and $a=\pm 2 p^{\frac{n}{2}}$.
\end{enumerate}

\bigskip

\item[(3.2)]  $p=3$:
\begin{center}
\begin{enumerate}
\item[(i)] $n \geq 3$ odd and $a=0$.
\item[(ii)]  $n$ even and $a$ verifies one of the following conditions:

\begin{center} $ \left( a=0 \right)$ \ \ \ or \ \ \ $\left( a=\pm p^{\frac{n}{2}} \right)$ \ \ \ or \ \ \ $\left( a =\pm 2 p^{\frac{n}{2}} \ \ \mbox{ and } \ \ 
n \geq 4  \right).
$
\end{center}
\end{enumerate} 
 \end{center}
 
\bigskip

\item[(3.3)] $p > 3$:

\begin{enumerate}
\item[(i)] $n$ even and $a=\pm 2 p^{\frac{n}{2}}$.
\item[(ii)] $n$ even, $p \not\equiv 1 \pmod{3}$ and $a=\pm p^{\frac{n}{2}}$.
\item[(iii)] $n$ odd and $a=0$.
\item[(iv)] $n$ even, $p \not\equiv 1 \pmod{4}$ and $a=0$.
\end{enumerate}

 \end{enumerate}

\noindent
In all these cases, $f(T)$ is the characteristic polynomial of an abelian surface defined over ${\mathbb F}_q$ which is isogenous to 
$E \times E$ where $E$ is a supersingular elliptic curve and its isogeny class contains the Jacobian of a Hallouin-Perret-maximal curve.
\end{small}

\end{theorem}

\begin{proof}
Throughout the proof we will consider a \textit{square} polynomial
$f(T)=(T^2+aT+q)^2=T^4+(2a) T^3 + (a^2+2q)T^2 + (2a) qT +q^2 \in{\mathbb Z}[T]$.
By R\" uck's result (see Theorem 1.1 of \cite{Ruck} or Lemma 2.1 of \cite{MNH})
we know that $f$ is a $q$-Weil polynomial if and only if 
 $\vert a\vert \leq 2\sqrt{q}$ and $4\vert a\vert \sqrt{q} \leq a^2+4q$.

(1) Let us first consider the case of simple abelian surfaces.
Theorem 2.9 in \cite{MNH} provides necessary and sufficient conditions 
for a $q$-Weil
polynomial $f(T)=T^4+a_1T^3+a_2T^2+qa_1T+q^2 \in{\mathbb Z}[T]$
to be the characteristic polynomial 
of a simple abelian surface defined over ${\mathbb F}_q$.

This theorem classifies such surfaces 
in four
families
 called mixed (M), ordinary (O) and supersingular (SS1) and (SS2).
This classification involves the integer $\Delta=a_1^2-4a_2+8q$,
but 
when $f(T)=(T^2+aT+q)^2$
we have $\Delta=4a^2-4(a^2+2q)+8q=0$.
In our case $\Delta$ is a square and so the cases (M) and (O) are discarded.
It is not possible to fulfill the conditions of the cases
(SS1) or (SS2), except when
$a=0$, $n$ is even and $p\equiv 1\pmod 4$ 
or when
$a=\pm \sqrt{q}$, $n$ is even and $p\equiv 1\pmod 3$.
For these values we check that $f(T)=(T^2+aT+q)^2$
does correspond to a Weil polynomial thanks to R\" uck's result.

Moreover, Howe, Nart and Ritzenthaler have given in Theorem 1.2 of
\cite{HNR} 
necessary conditions on a Weil polynomial 
to be 
the characteristic polynomial of a simple abelian surface defined over ${\mathbb F}_q$
which is not isogenous to a Jacobian.
Since 
 $a_2=a^2+2q$ 
the  condition $a_2<0$ of Table 1.2  of  \cite{HNR} is never satisfied.
 We conclude that 
in the considered cases 
$f(T)$ is indeed the characteristic polynomial of a simple abelian surface $A$ isogenous to a Jacobian.
This concludes the case (1) concerning simple abelian surfaces.

(2) If $(a,q)=1$  
then by Deuring and Waterhouse
the polynomial $T^2+aT+q$ is the
characteristic polynomial of an ordinary elliptic curve $E$ 
and thus $(T^2+aT+q)^2$ is the characteristic polynomial 
of an abelian surface isogenous to $E\times E$.
Moreover, by Theorem 2.3 of  \cite{HNR}, 
$E\times E$ is isogenous to the Jacobian of a curve $X$
 if and only if $a^2-4q$ is neither $-3$ nor $-4$ nor $-7$.
In this case, we know that $L_X(T)=(T^2+aT+q)^2$
and so $X$ is
indeed a
 HP-maximal curve by Proposition \ref{Alpha}.

(3) It remains to consider abelian surfaces 
which are isogenous 
to the product of two supersingular elliptic curves.

(3.1) We first treat the characteristic 2 case.
Maisner and Nart have given  in Table 1 
(respectively Table 3) in
\cite{M-N} 
the list of the 6 (respectively 15) isogeny classes 
of abelian surfaces in characteristic 2 that contain
the product of two supersingular elliptic curves over ${\mathbb F}_{2^n}$ 
when $n$ is odd (respectively even), 
together with
the numbers of ${\mathbb F}_{2^n}$-isomorphism 
classes of supersingular curves of genus 2 
 whose Jacobian
 lies
  in each isogeny class.
Theirs tables involve the couples $(a_1,a_2)=(2a,a^2+2q)$.

The cases in Table 1 and Table 3 for
which the condition $a_2 = \frac{a_1^2}{4}+2q$ holds are the following:
$$
\left\{
\begin{array}{cclll}
(a_1,a_2)&=&(0,2q) & \textrm{for odd } n > 1 & \textrm{(so } a=0),  \\
(a_1,a_2)&=&(\pm 2 \sqrt{2q},4q) & 
 \textrm{for odd }
n >1 & \textrm{(so } a=\pm  \sqrt{2q}), \\
(a_1,a_2)&=&(0,2q) &\textrm{for even } n & \textrm{(so } a=0), \\
(a_1,a_2)&=&(\pm 2\sqrt{q},3q) &\textrm{for even } n & \textrm{(so } a=\pm 
\sqrt{q}),  \textrm{ and} \\ 
(a_1,a_2)&=&(\pm 4 \sqrt{q},6q) &\textrm{for even }  n > 2 & \textrm{(so } a=\pm 
2\sqrt{q}).
\end{array}
\right.
$$
So when $n$ and $a$ fulfill one of these conditions
we know that there exists an elliptic curve $E$ such 
that $E \times E$ is isogenous to the Jacobian of a curve
and admits 
$f(T)=(T^2+aT+q)^2$ as characteristic polynomial.
Since $a$ and $p$ are not coprime, the elliptic curve $E$ is supersingular according
to Theorem 4.1 of \cite{Waterhouse}.

(3.2) We now deal with the characteristic 3 case.
In this context, 
Howe has determined 
all the polynomials that occur as 
characteristic polynomials 
of abelian surfaces and
whose isogeny class contains the Jacobian of a curve of genus $2$
(see Theorem 1.1 in \cite{Howe}).
We can identify among them
the only polynomials
which are square polynomials $f(T)=(T^2+aT+q)^2$: 
 $$
\left\{
\begin{array}{lll}
f(T)=(T^2+q)^2 & \textrm{for odd } n \geq 3 & \textrm{(so } a=0),  \\
f(T)=(T^2+q)^2 & \textrm{for even }n  & \textrm{(so } a=0), \\
f(T)=(T^2\pm \sqrt{q}T+q)^2&\textrm{for even } n & \textrm{(so } a=\pm \sqrt{q}), \\
f(T)=(T^2\pm 
2\sqrt{q}T+q)^2
 &\textrm{for even } n \geq 4 & \textrm{(so }
a=\pm  2 \sqrt{q}). \\ 
\end{array}
\right.
$$
What is left is to show
that in any case 
the corresponding abelian surface is not simple
and that 
$T^2+aT+q$ is indeed
the characteristic polynomial of a supersingular elliptic curve.
To show that the abelian surface is not simple, 
it is sufficient to check that none of the conditions {\sl (1.1)} and {\sl (1.2)} of the theorem is satisfied.
Now, according to 
Deuring-Waterhouse theorem
(Theorem 4.1
in \cite{Waterhouse}), there is a natural bijection between 
the isogenous classes of elliptic curves 
over ${\mathbb F}_q$ and the set of integers $m$ 
with $\lvert m \lvert \leq 2\sqrt{q}$
 that satisfy one of the following conditions:
\begin{enumerate}
\item $(q,m)=1$;
\item $q$ is a square and $m=\pm2\sqrt{q}$;
\item $q$ is a square, $p\not\equiv 1\pmod 3$, and $m=\pm\sqrt{q}$;
\item $q$ is not a square, $p=2$ or $3$, and $m=\pm\sqrt{pq}$;
\item $q$ is not a square and $m=0$; or $q$ is a square, $p\not\equiv 1\pmod 4$, and $m=0$.
\end{enumerate}
And this bijection is such that
$\sharp E({\mathbb F}_q)=q+1-m$ for all curves $E$ from the isogeny class corresponding to $m$.
Furthermore
only the first item does not correspond to supersingular elliptic curves.
So 
we just have
 to check that $\vert a \vert \leq 2 \sqrt{q}$ and 
 that 
 the conditions of one of the cases (2), (3), (4) or (5) are satisfied.


(3.3) The case of characteristic $p>3$ will be a conclusion
of works undertaken  by Howe, Nart and Ritzenthaler. 
Indeed, in this setting,
Theorem 2.4 of \cite{HNR} asserts that 
there is a Jacobian isogenous 
to the product of two supersingular elliptic 
curves $E_1$ and $E_2$ 
of characteristic polynomials $f_{E_1}(T)=T^2-rT+q$
and $f_{E_2}(T)=T^2-sT+q$ if and only if $r^2=s^2$.
When $E_1=E_2$ this condition is certainly satisfied.
Furthermore, if the Jacobian of a curve $X$ is isogenous to $E \times E$,
then the $L$-polynomial of $X$ has form 
$L_X(T)=(1-2\alpha T + qT^2)^2$ 
which implies
by Proposition 2.1 that $X$ is HP-maximal.

So 
we are reduced to discuss under which conditions there exists a supersingular elliptic curve $E$.
The answer is given
by the Deuring-Waterhouse theorem
 (Theorem 4.1 of \cite{Waterhouse}) already quoted above,
namely by the conditions (2), (3), (4) and (5) of the list given in point (3.2).
\end{proof}

Recall that any  curve of genus 1 is a HP-maximal curve (see Remark \ref{remark:elliptic}).
As a consequence, there exists a HP-maximal curve 
whatever the base  field.
What about the question for non-elliptic curves?
The previous theorem brings   the following answer.

\begin{proposition}\label{prop:HPexistence}
Over any finite field   there exists a non-elliptic Hallouin-Perret-maximal curve.

More precisely, 
there exists a Hallouin-Perret-maximal curve of genus 2 defined over ${\mathbb F}_q$ if and only if $q>2$,
and
 there exists a Hallouin-Perret-maximal curve of genus 3 over ${\mathbb F}_2$.
\end{proposition}

\begin{proof}
Maisner, Nart and Howe have provided in \cite{MNH}
a complete description  
of the curves of genus 2  defined over ${\mathbb F}_2$
up to ${\mathbb F}_2$-isomorphism 
and quadratic twist. 
None of the couples $(a_1,a_2)$ given in Table 2 in \cite{MNH} 
satisfies the condition $a_2=a_2/4+2q$,
so there is no  HP-maximal curve of genus 2 over ${\mathbb F}_2$.

Now we assume that $q>2$,
and we first consider the case of characteristic 2.
When $q=4$, Table 7  
of \cite{MNH} shows (first row and last column) 
the existence of a curve
such that $a_1=0$ and $a_2=8=\frac{a_1^2}{4}+2q$ 
so the existence of a HP-maximal curve of genus 2 over ${\mathbb F}_4$
is established.

If $q=2^n$ with $n\geq 3$,
the choice $a=3$ gives
$\vert a\vert\leq2\sqrt{q}$ 
and fulfills the conditions $(a,p)=1$ and $a^2-4q\not\in\{-3,-4,-7\}$
of point (2) of Theorem  \ref{HP_curves_of_genus_2}. 
Thus the polynomial $(T^2+3T+q)^2$
is the characteristic polynomial
of an abelian surface 
(isogenous to $E\times E$ where $E$ is an ordinary elliptic curve) 
whose isogeny class contains the Jacobian of a HP-maximal curve of genus 2 defined over ${\mathbb F}_q$.

Finally 
in the case of characteristic  greater than or equal to $3$
one can take $a=2$, so once again the conditions 
of point (2) of Theorem  \ref{HP_curves_of_genus_2} are fulfilled,
and so there exists a HP-maximal curve of genus 2 defined over ${\mathbb F}_q$.

For the remaining case $q=2$, 
the LMFDB database (see \cite{LMFDB})  provides a genus-$3$ HP-maximal curve,
namely\footnote{\url{https://www.lmfdb.org/Variety/Abelian/Fq/3/2/ad_j_an}}
the pointless curve defined
by the equation 
$x^4+x^2y^2+x^2yz+x^2z^2+xy^2z+xyz^2+y^4+y^2z^2+z^4=0$ 
which admits $14$ points over $\mathbb{F}_{4}$.
\end{proof}

{\bf Acknowledgement.} The authors would like to thank Christophe Ritzenthaler
for useful comments
 and for pointing out 
 the curve of Remark \ref{remark:christophe}.
 They are also grateful to the referee for the careful reading and useful comments.

\printbibliography
\end{document}